\documentclass[12pt,english]{amsart}
\usepackage[]{fontenc}
\usepackage[latin1]{inputenc}
\usepackage{geometry}
\usepackage{fancyhdr}
\usepackage{graphicx}
\usepackage{setspace}
\makeatletter
 \theoremstyle{plain}
\newtheorem{thm}{Theorem}[section]
  \theoremstyle{remark}
  \newtheorem{rem}[thm]{Remark}
  \theoremstyle{plain}
  \newtheorem{cor}[thm]{Corollary}
  \theoremstyle{plain}
  \newtheorem{prop}[thm]{Proposition}
  \theoremstyle{plain}
  \newtheorem{lem}[thm]{Lemma}
  \theoremstyle{plain}
  \newcounter{clm}[thm]
  
  \theoremstyle{plain}
  \newtheorem{dfn}[thm]{Definition}
  \theoremstyle{remark}
  \newtheorem{ex}[thm]{Example}

\usepackage[all]{xy}
\usepackage{amssymb}
\usepackage{pstricks}
\usepackage{pst-node}
\usepackage{pst-plot}

\def\al{\alpha}

\def\ga{\gamma}
\def\de{\delta}

\def\la{\lambda}
\def\La{\Lambda}
\def\si{\sigma}
\def\th{\vartheta}

\def\ze{\zeta}
\def\om{\omega}
\def\Om{\Omega}
\def\dd{{\mathrm d}}
\def\e{{\mathrm e}}
\def\ii{{\mathrm i}}
\def\t{\,{}^{t}\!}
\def\C{{\mathbb C}}

\def\DD{{\mathbb D}}

\def\H{{\mathbb H}}
\def\N{{\mathbb N}}

\def\P{{\mathcal P}}
\def\Q{{\mathcal Q}}
\def\QQ{{\mathbb Q}}
\def\R{{\mathbb R}}
\def\S{{\mathcal S}}
\def\T{{\mathbb T}}
\def\Z{{\mathbb Z}}
\def\pii{\pi\mathrm{i}}
\def\im{\mathrm{Im}\,}

\def\Pic{\mathrm{Pic}}

\def\Pol{\mathrm{Pol}}

\newcommand{\clth}[2]{\vartheta\left[\begin{array}{@{}c@{}}
      #1 \\ #2 \end{array}\right]}
\def\tns{\otimes}

\def\ol{\overline}

\newcommand{\pd}[2]{\frac{\partial #1}{\partial #2}}

\newcommand{\cdiag}[1]{\begin{array}{c} \xymatrix{ #1 } \end{array}}
\def\mod{\rm{mod}\,}

\newcommand{\be}{\begin{equation}}
\newcommand{\ee}{\end{equation}}
\newcommand{\ba}{\begin{eqnarray}}
\newcommand{\ea}{\end{eqnarray}}
\def\nn{\nonumber}

\def\sign{\mathrm{sign}\,}
\newcommand{\vect}[1]{\left[\begin{array}{@{}c@{}} #1 \end{array}\right]}
\newcommand{\mtrx}[1]{\left[\begin{array}{@{}cc@{}} #1 \end{array}\right]}

\usepackage{babel}
\makeatother
\begin{document}

\setcounter{clm}{0}
\renewcommand{\theclm}{\arabic{clm}}

\title{Quantization of Abelian Varieties: distributional sections and the transition from K\"ahler
to real polarizations}

\author{Thomas Baier, Jos\'e M. Mour\~ao and Jo\~ao P. Nunes}

\email{tbaier, jmourao, jpnunes@math.ist.utl.pt}

\address{Department of Mathematics, Instituto Superior T\'{e}cnico, Av. Rovisco
Pais, 1049-001 Lisboa, Portugal}

\begin{abstract}
We study the dependence of geometric quantization of the standard symplectic
torus on the choice of invariant polarization. Real and mixed polarizations
are interpreted as degenerate complex structures. Using a weak version of the
equations of covariant constancy, and the Weil-Brezin expansion to describe
distributional sections, we give a unified analytical description of the
quantization spaces for all nonnegative polarizations.

The Blattner-Kostant-Sternberg (BKS) pairing maps between half-form corrected 
quantization spaces for different polarizations are shown to be transitive and 
related to an action of $Sp(2g,\R)$. Moreover, these maps are shown to be
unitary.

\end{abstract}

\maketitle

\begin{center}
\today
\end{center}

\tableofcontents{}

\section{Introduction}

Abelian varieties provide a very rich example of interplay between 
ideas of algebraic geometry and of geometric quantization. A general problem of great 
interest in the geometric quantization of a large class of symplectic manifolds 
is the dependence of quantization on the complex structure \cite{AdPW,Hi}. (See also \cite{An}.)

As is well known, the holomorphic quantization of a symplectic torus, once it is equipped with 
the structure of abelian variety, produces the space of theta functions. 
In this case, quantizations in different complex structures yield
spaces of theta functions which are naturally unitarily equivalent.
This equivalence between holomorphic quantization spaces has been obtained in \cite{AdPW}
by means of the parallel transport induced by a heat equation. See also 
\cite {FMN1,FMN2} for an intrinsically finite-dimensional approach.

In this work, we consider the geometric quantization 
of the standard symplectic torus of dimension $2g$, $(\T^{2g},\om)$, in 
holomorphic and real nonnegative invariant polarizations and for any level $k\in\N$. 
Due to the group structure of $\T^{2g}$, and consequent triviality of the
tangent bundle, the space of invariant polarizations is canonically identified with the 
nonnegative Lagrangian Grassmannian of a fixed symplectic vector space, $\C^{2g}$.
Thus, one obtains some similarities with the geometric quantization of a vector space \cite{AdPW,KW}.
For instance, below we construct a Blattner-Kostant-Sternberg (BKS) pairing, 
between quantization spaces associated with invariant nonnegative polarizations, which
is transitive, once the half-form correction is included.

The main novel point in our approach is the treatment of the equations of covariant 
constancy of wave functions, for real and holomorphic polarizations, from a unified
analytical stand point, considering the operators of covariant derivation to act on
distributional sections of the prequantum line bundle. For real polarizations this
yields, as expected, distributional sections supported on Bohr-Sommerfeld fibers; 
still, it does not always produce the same result as the more traditional cohomological
wave function approach to quantization of \cite{Sn}, as can be seen in the case of
toric manifolds comparing \cite{BFMN} and \cite{Ham}.
For previous related work on the geometric quantization of abelian varieties see \cite{An}, 
where, in particular, it is explicitly shown that the quantization spaces for any 
two invariant real polarizations are isomorphic.

In the language of geometric quantization, the space of invariant complex structures
corresponds to the space of positive Lagrangian subspaces, which is usually identified
with Siegel upper-half space $\H_g$. To allow the study of the dependence of
quantization on the complex structure, including the case of degenerate complex
structures on the boundary of $\H_g$, it is convenient to consider also another chart
that uses the closed Siegel
disc, which we denote by $\DD_g$. In this way, we obtain a global parameter also for the
nonpositive definite part (that becomes identified with $\partial\DD_g$) of the
Lagrangian Grassmannian.

A crucial contribution for the nice behaviour of quantization (or rather of the relation
between any two quantizations) on the boundary of the upper-half space is given by the
half-form correction. As in many other examples in geometric quantization \cite{Wo}, the
half-form is necessary for the unitarity of the pairing maps relating quantizations in
different polarizations.

Symmetry groups play an important role in this work. On one hand, the BKS pairing between
different invariant polarizations is intimately related with the 
action of $Sp(2g,\R)$ on $\DD_g$. This action does not represent a geometric 
symmetry of the manifold we are quantizing, but has, instead, an analytical flavour 
connected to the natural representation of the metaplectic group on $L^2(\R^g)$.
Also, one of the most interesting feature of the geometric quantization of $\T^{2g}$, that
is not present in the linear case  $\R^{2g}$ (see \cite{KW}), besides the appearance of
distributional sections of quite different nature over different degenerate Lagrangian
subspaces, is the interplay with natural invariance groups of certain data of the
prequantization of the torus. Namely, there are natural 
geometric actions of the integer symplectic group, $Sp(2g,\Z)$, and of 
the group $(\Z/k\Z)^{2g}$ on $\T^{2g}$ that leave the holonomies of the prequantum line bundle
representing $k\om$ invariant. While the latter gives rise to the finite
Heisenberg group when lifted to the line bundle, the former gives origin, for
non-degenerate complex structures, to the classical algebro-geometric theta-transformation
formula (see, for example, \cite{BL,Po,Ke}). We will address some of these issues in \cite{BMN}.

In \cite{Ma}, Manoliu shows (for even level $k$) that the BKS pairing maps relating two
reducible real polarizations is unitary if one takes into account the half-form correction.
In this work, where we consider arbitrary level $k$, we obtain an analogous 
result for all
holomorphic polarizations, while the real polarizations are included as limiting cases on the
boundary of the space of complex structures.

Abelian varities, therefore, give one more family of symplectic manifolds, in 
addition to non-compact complex Lie groups (see \cite{Hal,FMMN}), where half-form
corrected holomorphic quantizations in different complex structures, including real
polarizations as degenerate cases, can be related by unitary BKS pairing maps \cite{Ra,Wo}.

The paper is organized as follows. After some preliminaries in Section \ref{av_prelim},  
we use the $Sp(2g,\R)$ action in Section \ref{av_dcqb} to give a unified analytical description of half-form 
quantization 
in complex, real, or mixed polarizations, resulting in Theorem \ref{av_teorema}. 
In Section \ref{avbks}, we show that the BKS pairing maps are unitary and transitive.
Section \ref{tropic} contains a brief illustration of how tropical geometry can be seen to emerge
as the complex struture degenerates.

\section{Preliminaries on the prequantum line bundle and its sections}

\label{av_prelim}

\subsection{The prequantum line bundle $L$}

Let $(\T^{2g}=\R^{2g}/\Z^{2g},\om)$ be the standard even-dimensional torus, 
with periodic coordinates $(x,y)$ with $x,y\in \R^g$, and the invariant
symplectic form given by $\om = \sum_{i=1}^g \dd y_i \wedge \dd x_{i}$. 
Consider a prequantization of $\T^{2g}$ given by the line bundle $L$,
representing the cohomology class of the symplectic form, with Hermitian
structure and a compatible connection with curvature $-2\pii \om$. $L$
is defined by  $L=\R^{2g}\times_{\Z^{2g}}\C$, where the $\Z^{2g}$ action is
\be
\label{av_autom}
\la\cdot (u,\ze) = \left( u+\la,\al(\la)\e^{-\pii \om(u,\la)}\ze \right) ,\quad
 u=\vect{x\\y}\in\R^{2g}, \la\in \Z^{2g},
\ee
where $\al:\Z^{2g}\to \{\pm 1\} \subseteq U(1)$ is the so-called ``canonical''
semi-character
\[
 \al(\la) = (-1)^{\sum_{j=1}^g \la_j \la_{g+j}}
\]
and $\om$ is identified with symplectic bilinear form
\[
 \om\left(u,v\right) = \sum_{j=1}^g (u_{g+j}v_j-u_j v_{g+j}) 
 = \t u \mtrx{0 &-I\\I&0} v .
\]

The Hermitian structure and connection are defined to be,
\be
\label{av_herm}
 h((u,\ze),(u,\ze')) = \ze\ol{\ze'}
\ee
and
\be
\label{av_conn}
 \nabla s = \dd s - \pii s\,\sum_{j=1}^g \left( y_{j}\dd x_j
 - x_j\dd y_{j} \right) ,
\ee
respectively. Here and below, global sections of $L$ 
(for example smooth ones, $s\in \Gamma^\infty(L)$) will be identified 
with sections of the (trivialized) pull-back of $L$ to $\R^{2g}$, that is, with 
functions on $\R^{2g}$ that satisfy the appropriate quasi-periodicity
conditions,
\[
 s(u+\la) = \al(\la)\e^{-\pii \om(u,\la)}s(u),\quad \la\in \Z^{2g} .
\]
Since the trivialization of the bundle is unitary and the connection
1-form purely imaginary, the connection is Hermitian.

\subsection{Spaces of sections of $L^k$}

If $s$ is a smooth section of $L^k$, then $$\e^{k\pii \t x y}s(x,y)$$ is periodic
in $x$ and hence admits a Fourier expansion. This observation leads to an
isomorphism given by the Weil-Brezin expansion (also known as Zak expansion)
\cite{F} given by,
\begin{eqnarray} \label{av_wb}
 \Gamma^\infty(L^k) & \longrightarrow & \prod_{l\in(\Z/k\Z)^g} \S(\R^g)  \\
 s & \mapsto & (s)_l(y) := \int_{[0,1]^g} s(x,y+\frac{l}{k}) \e^{k\pii \t x(y+\frac{l}{k})}
 \e^{-2\pii \t l x} \dd^g x \nn
\end{eqnarray}
with inverse
\begin{eqnarray}
\label{av_wb2}
 \prod_{l\in(\Z/k\Z)^g} \S(\R^g) & \longrightarrow & \Gamma^\infty(L^k)  \\
 \left\{ (s)_l \right\}_{l\in (\Z/k\Z)^g} & \mapsto & s(x,y) := \e^{-k\pii \t x y}
 \sum_{l\in (\Z/k\Z)^g} \sum_{m\in \Z^g} (s)_l(y-m-\frac{l}{k}) \e^{2\pii \t (km+l) x} \nn
\end{eqnarray}
where $0\leq l_j < k$. We will use the round bracket notation $(s)_l$ for the Weil-Brezin
coefficients of a section $s$ throughout the paper.

The Weil-Brezin map is an isomorphism between topological vector spaces (of smooth
sections and Schwartz functions, respectively), hence it extends to an isomorphism
between the dual spaces, that is between the space of distributional sections of
the bundle $L^k$ and a product of $k^g$ copies of the space of tempered
distributions $\S'(\R^g)$. This map in turn restricts to a unitary isomorphism
$\prod_{l\in (\Z/k\Z)^g} L^2(\R^g) \leftrightarrow \Gamma_{L^2}(L^k)$, with
\be
 \label{av_unit}
 \langle s,s'\rangle = \sum_{l\in (\Z/k\Z)^g} \langle (s)_l,(s')_l \rangle = 
 \sum_{l\in (\Z/k\Z)^g} \int_{\R^{g}} (s)_l \ol{(s')_l}.
\ee

\subsection{Invariant complex structures and theta functions}

Invariant complex structures on $\T^{2g}$ are determined by their restriction
to the tangent space at any point and hence can be parametrized by matrices in
Siegel upper half-space $\H_g$ (see, for instance, \cite{Ke}). The complex
structure on $\T^{2g}$ can be described by its lift to the universal cover
$\R^{2g}$: for any $\Om\in\H_g$, let $\Lambda_\Om$ be the lattice $\Z^{g}\oplus
\Om \Z^{g}\subset\C^g$, so that the torus is equipped with the structure of an
Abelian variety via the smooth isomorphism $\phi_\Om:\T^{2g}\to X_\Om :=
\C^g/\Lambda_\Om$ induced by ($x,y\in\R^g$)
\be
\label{av_realiso}
 \R^{2g}\ni u = (x,y) \mapsto z_\Om := x-\Om y \in \C^g .
\ee
(The choice of the sign here, as well as for the symplectic form $\om$ were made
so that the action of the symplectic group turns out to be the expected one in
all the coordinates). In the coordinates $x,y$, the complex structure takes the
form
\be\label{av_JOm}
 J_\Om = \mtrx{ -\Om_1\Om_2^{-1} & \Om_1\Om_2^{-1}\Om_1+\Om_2 \\
                -\Om_2^{-1} & \Om_2^{-1}\Om_1 } , \quad
 \Om = \Om_1 + \ii \Om_2,
\ee
but we will need this only in the last Section.
The corresponding holomorphic structure on $(L,\nabla)$ is given by the action
of the lattice $\Lambda_\Om$ on $\C^g\times \C$ obtained by combining
(\ref{av_autom}), (\ref{av_realiso}) and the isomorphism
\[
 \R^{2g}\times\C\ni (u,\ze) = ((x,y),\ze) \mapsto (x-\Om y,
 \e^{\pii F_\Om(z_\Om,z_\Om)}\ze) \in \C^g\times\C.
\]
We will denote this holomorphic line bundle by
$L_{\Om}:=\C^g\times_{\Lambda_\Om}\C$. Here $F_\Om$ is the bilinear
form
\[
 F_\Om(z,w) = -\t z(\im\Om)^{-1}\im w
\]
or, in terms of real coordinates, $z=x-\Om y$, $w=u-\Om v$,
\[
 F_\Om(x-\Om y,u-\Om v) = \t(x-\Om y)v = \t x v - \t y \Om v .
\]
The quasi-periodicity condition for functions on $\C^g$ to define
holomorphic sections of $L_\Omega$ takes the well-known form 
corresponding to classical theta functions,
\[
 \th(z+\la) = \al(\la)\e^{2\pii F_\Om(z,\la)+\pii
   F_\Om(\la,\la)}\th(z),\,\,\,\lambda\in \Lambda_\Om.
\]

It is well known (see, for example, \cite{Ke,Po}) that a basis of
$H^0(X_\Om,L_{\Om}^k)$ is given by
\be \label{av_classic}
 \left\{\clth{\frac{l}{k}}{0}(kz,k\Omega)\right\}_{l\in\Z^g/k\Z^g},
\ee
which corresponds to the sections
\ba
 \th^l_\Om(x,y) & := & \e^{-k\pii F_\Om(x-\Om y,x-\Om y)}
 \clth{\frac{l}{k}}{0}(kx-k\Om y,k\Om) \ = \nonumber \\
 \label{av_thetas} 
 & = & \e^{-k\pii\t x y} \sum_{m\in\Z^g}
 \e^{k\pii\t(y-m-\frac{l}{k})\Om(y-m-\frac{l}{k}) +
   k2\pii\t(m+\frac{l}{k})x} 
\ea
of $L^k$. Note that the Weil-Brezin coefficients of these
sections are the Gaussians
$$
(\th^l_\Om)_{l'}(y)= \delta_{l l'} \e^{k\pii \t y\Omega y}.
$$

\subsection{Invariant polarizations}

\label{av_tau}
Geometric quantization (for example, of the symplectic torus) can be performed
not only using the K\"ahler structures from the last paragraph, but using the
more general notion of a polarization. Here, we describe all invariant
polarizations on $\T^{2g}$.

Let $L\T^{2g}\to \T^{2g}$ be the bundle of nonnegative Lagrangian subspaces of
the complexified tangent bundle on the symplectic manifold $(\T^{2g},\omega)$.
A (nonnegative) polarization, in the sense of geometric quantization (see, for
instance, \cite{Wo}), is a section of $L\T^{2g}$ that provides an involutive
distribution on $\T^{2g}$,
\[
 \Pol(\T^{2g}) \subset \Gamma^\infty(L\T^{2g}) .
\]
In the present case, $L\T^{2g}$ is canonically trivial
\[
 L\T^{2g} \cong \T^{2g}\times L \C^{2g} ,
\]
(where $L \C^{2g}$ denotes the Grassmannian of nonnegative Lagrangian subspaces
in $\C^{2g}$ equipped with the standard symplectic form), since $T\T^{2g}$ is
(canonically) trivial. A polarization $\P$, therefore, is simply a function
$\P:\T^{2g} \to L\C^{2g}$, and it is invariant if this function is constant,
\[
 \Pol^{\T^{2g}}(\T^{2g}) \cong L\C^{2g} .
\]

The complex structures from the previous section, parametrized by matrices $\Om\in\H_g$
in Siegel upper half-space via the coordinates $z_\Om=x-\Om y$, correspond to the positive
(or K\"ahler) polarizations $L^+\C^{2g}$ in $\Pol^{\T^{2g}}(\T^{2g})$,
\[
 \P_\Om = \mathrm{span}\left\{\frac{\partial}{\partial
   {z}_\Om^j} \right\}_{j=1,\dots, g} =
 \mathrm{span}\left\{\sum_l \left( \ol{\Om}_{l k}\pd{}{x_l}+\pd{}{y_k} \right)
 \right\}_{k=1,\dots, g} .
\]

\begin{rem} Notice that the positive Lagrangian Grassmannian bundle $L^+\T^{2g}
\cong \T^{2g}\times\H_g$ carries also a natural complex structure: equipping every
fiber $\T^{2g}\times\{\Om\}$ with the complex structure from the isomorphism $\T^{2g}
\cong \C^g / \Lambda_\Om$ gives the universal bundle of Abelian varieties (with a
marked basis of the first homology) over Siegel upper half-space.
\end{rem}

This coordinate chart $\H_g \to \Pol^{\T^{2g}}(\T^{2g})$ is not convenient for the
description of all genuine nonnegative polarizations (namely, only those transverse
to the polarization spanned by the $\pd{}{x_j}$ directions appear in the closure of
$\H_g$ as a space of matrices). The convenient substitute (see, for instance, \cite{Ki})
for it is the closed Siegel disc $\DD_g$, which is the closure of the image of
$\H_g$ under the Cayley transform
\begin{equation} \label{av_tausomegas}
 \H_g \ni \Om \mapsto  \tau = (\ii-\Om)(\ii+\Om)^{-1} \in \breve{\DD}_g ,
\end{equation}
with inverse $\Om = \ii (1-\tau)(1+\tau)^{-1}$. The Siegel disc is a global chart
of $L\C^{2g}$, and the parametrization of invariant polarizations now reads
\be\label{av_polll}
 \P_\tau = \mathrm{span}\left\{\sum_l \left( -\ii (1-\ol{\tau})_{l k}\pd{}{x_l}
 +(1+\ol{\tau})_{lk}\pd{}{y_l} \right) \right\}_{k=1,\dots, g} .
\ee
Throughout the paper, $\tau$ will always denote a point in the closed Siegel disc, while
$\Om$ denotes a point in the upper half-space.

\begin{rem} A polarization is real if and only if $\tau$ is unitary. Following \cite{Wo},
we will call a real polarization {\it reducible}, if its space of leaves is a Hausdorff
manifold (or equivalently, in our case, if the leaves are compact). Note that this happens
if and only if $\tau$ is unitary and its entries lie in $\QQ[\ii]$. Reducible polarizations
and the quantizations defined by them
(whose elements are supported on leaves along which $\nabla$ has trivial holonomy,
the so-called Bohr-Sommerfeld leaves)
have been studied in the present context using geometric methods (see \cite{Ma} and
also earlier work by \'{S}niatycki \cite{Sn}), and will also be investigated in more detail
below.
Both reducible and non-reducible polarizations are also considered in \cite{An}.
\end{rem}

\subsection{The action of the real symplectic group}

Below, we will consider an action of the real
symplectic group $Sp(2g,\R)$ on the Lagrangian Grassmannian (see \cite{Wo}), which is
intimately related to the construction of the half-form (or metaplectic) correction
and the Blattner-Kostant-Sternberg pairing. On the upper half-space chart on the
positive Lagrangian Grassmannian, the action is given by the usual 
fractional linear transformation
\be
\label{av_mom}
 M(\Om) = (A\Om+B)(C\Om+D)^{-1}, \quad M = \left[ \begin{array}{cc} A & B \\
 C & D \end{array} \right] \in Sp(2g,\R) ,
\ee
and it is transitive on $\H_g$. It extends to all of $L\C^{2g}$ (or $\DD_g$), where
it becomes transitive on each stratum; in terms of the parameter $\tau$ and setting
$\tau':=M(\tau)$, the transformation is characterized by the equation
\begin{equation} \label{av_changetau}
 \big( (\tau'+1)(A+\ii B) + (\tau'-1)(\ii C-D) \big) \tau = (\tau'+1)(A-\ii B) +
 (\tau'-1)(\ii C+D) .
\end{equation}
In particular, the action of the integer symplectic group $Sp(2g,\Z)$ by
symplectomorphism of $\T^{2g}$
\[
 \T^{2g} \ni \vect{x\\y} \mapsto M \vect{x \\ y} \in \T^{2g},
\]
induces an action on polarizations via push-forward given by
\[
 M_*\P_\Om = \P_{M(\Om)} .
\]

\begin{rem} If we write this action on the symplectic coordinates $(x,y)$ in terms of 
the complex ones, we recover the usual transformation
\[
 z_\Om \mapsto z_{\Om'} = \t(C\Om+D)^{-1}z_\Om .
\]
\end{rem}

The metaplectic group (discussed in more detail in Section \ref{av_weilrep} below) is
defined as the connected two-fold covering group
$Mp(2g,\R)\to Sp(2g,\R)$. It cannot be realized as a matrix group. A convenient
description for our purpose (see \cite{We,G}) involves an open subset
$U\subset Sp(2g,\R)$ that generates the group in two steps, i.e. $U^2 =
Sp(2g,\R)$; $U$ is parametrized by triples $(P,L,Q)$ of $g\times g$ matrices,
with $L$ invertible, $P$ and $Q$ symmetric, via
\[
 (P,L,Q) \mapsto \left[ \begin{array}{cc} PL^{-1} & PL^{-1}Q-\t L \\
 L^{-1} & L^{-1}Q \end{array} \right] .
\]
The usefulness of this subset $U$ for the description of the metaplectic group
lies in the fact that its pre-image $\widetilde{U}$ in $Mp(2g,\Z)$ consists
of two disjoint diffeomorphic copies of $U$. This gives a workable decription of
the metaplectic group, as we will see below.

\section{A distributional construction of the quantum bundle}
\label{av_dcqb}

In order to treat quantizations using complex, mixed and real polarizations in a uniform
way, it is convenient to widen our perspective on the prequantum Hilbert space $\Gamma_{L^2}(L^k)$
of square-integrable sections, and consider instead the whole Gelfand triplet
\[
 \Gamma^\infty(L^k) \subseteq \Gamma_{L^2}(L^k) \subseteq \Gamma^{-\infty}(L^k) .
\]
Once we take into account the factors arising from the half-form correction, the
pairing on the prequantum Hilbert space $\Gamma_{L^2}$ and its extension to the
Gelfand triplet will, by construction, provide the BKS pairing, as we will see in the
next Section. The fact that the pairing splits nicely into a ``prequantum factor''
and a ``half-form factor'', as in the vector space case, is ultimately a consequence of the group
structure on $(\T^{2g},\om)$. Before studying the pairing, we will define
the quantum bundle (or rather, interpret the usual definition, as in \cite{Wo} or
\cite{Ki}) in the setting just outlined.

\begin{rem} Similar analytic behaviour occurs in the quantization of a flat symplectic
vector space, where in particular it also turns out to be necessary to include half-form
correction to obtain continuous behaviour of the BKS pairing up to the boundary
\cite{KW}. Also, in the context of toric varieties, a similar convergence to distributional
sections occurs when holomorphic polarizations degenerate to the toric real polarization,
with fibers the compact Lagrangian tori \cite{BFMN}.
\end{rem}

Examining the set of differential equations attached to a choice of polarization
that singles out the subspaces of covariantly constant sections, we will see that they
actually admit a natural weak variant that coincides with the usual equations for
square-integrable sections in the case of a complex polarization, but with the advantage
that there is a (non-zero) space of solutions also for mixed and real polarizations.
It is our aim, in this section, to show that this definition is the natural one from
the point of view of the transition from complex to mixed or real polarizations. This is
achieved in Theorem \ref{av_teorema} below.

\subsection{The weak equations of covariant constancy} \label{av_weakcc}

Given an invariant polarization $\P$, a smooth, at first, local section $s$ is covariantly constant along $\P$ if
\be \label{av_cc}
 \forall \xi \in \P: \quad \nabla_{\ol{\xi}} s =  \dd s \cdot \ol{\xi}
 -k\pii s(x,y) \sum (y_j \dd x_j -x_j \dd y_j) \cdot \ol{\xi} = 0 .
\ee
Note that if the polarization is given by a complex structure, these are just the Cauchy-Riemann
equations. In any case, every $\xi\in\P$ defines a continuous linear operator $\nabla_{\ol{\xi}}: 
\Gamma^\infty(L^k) \to \Gamma^\infty(L^k)$, and the quantum space associated with a complex
polarization $\P$ is given by the intersection of the kernels of these operators.

If the polarization has real directions (and the manifold we are quantizing is compact),
there are no non-zero global smooth solutions to the equations (\ref{av_cc}), since any such
solution would have to be supported on Bohr-Sommerfeld leaves:
smooth sections can be restricted to leaves, and non-trivial holonomy along any
non-contractible loop in a leaf forbids the existence of non-zero horizontal sections.

It is, therefore, natural to consider
the weak version of the operators (\ref{av_cc}) acting on distributional sections of $L^k$.
Distributional sections in the quantization of abelian varieties appeared also in \cite{An}.
As already mentioned in the introduction, in the case of toric manifolds this approach
has the advantage of including contributions from all, even singular, Bohr-Sommerfeld
fibers \cite{BFMN}; thus, the dimension of the quantization space in the real toric
polarization equals the dimension obtained from K\"ahler polarization.

\begin{rem}
\label{localtoric}
Locally, the neighbourhood of a Bohr-Sommerfeld fiber of $\T^{2g}$ is equivariantly 
symplectomorphic to a neighbourhood of a non-singular Bohr-Sommerfeld fiber of a toric manifold.
As seen in \cite{BFMN}, each Bohr-Sommerfeld fiber will be the support of a one-dimensional space of 
polarized distributional sections. Moreover, the explicit form of these sections is described 
in \cite{BFMN}.   
\end{rem}

{}For an open subset $U\subset \T^{2g}$, consider first the natural injection 
\ba\nonumber
\iota :\Gamma^\infty(L^k\vert_U) & \to & \Gamma^{-\infty}(L^k\vert_U)=(\Gamma^\infty_c(L^k\vert_U))'\\
\nonumber
s & \mapsto & \iota s (\phi) = \int_U s\phi \frac{\omega^n}{n!}. 
\ea
Then, the definition should make the following diagram commute,
\[
 \cdiag{
 \Gamma^{\infty}(L^k) \ar@{^{(}->}[r]^{\iota} \ar[d]_{\nabla_{{\xi}}} &
 \Gamma^{-\infty}(L^k) \ar[d]^{\nabla'_{{\xi}}} \\
 \Gamma^{\infty}(L^k) \ar@{^{(}->}[r]^{\iota} &
 \Gamma^{-\infty}(L^k) \\
} ,
\]
so that the operator $\nabla'_{{\xi}}$ extends the operator $\nabla_{{\xi}}$ to distributional sections.
Explicitly, on any open set $U$, for any smooth section $s\in \Gamma^\infty_c(L^k\vert_U)$ and for any test section $\phi \in \Gamma^\infty_c(L^{-k}\vert_U)$
with compact support and smooth section $\xi\in\Gamma^\infty(\P\vert_U)$ of the polarization on $U$,
\begin{eqnarray*}
(\nabla'_\xi \iota s) (\phi)& = & \int_U (\nabla_{{\xi}} s) \phi\, \dd^g x \dd^g y\\
& = & \int_U \left(\dd s \cdot {\xi}-k\pii s \sum (y_j \dd x_j -x_j \dd y_j) \cdot {\xi}\right) \phi\, \dd^g x \dd^g y\\
 & = & - \int_U s \left( {\rm div}\, {\xi} \phi + \dd \phi \cdot {\xi}+ k\pii \phi
 \sum (y_j \dd x_j -x_j \dd y_j) \cdot {\xi} \right)\, \dd^g x \dd^g y \\
 & = & -\int_U s \left(  {\rm div}\, {\xi} \phi + \nabla_{{\xi}}^{-1}\phi\right )\,\dd^g x \dd^g y,
\end{eqnarray*}
where $\nabla^{-1}$ stands for the connection on the inverse bundle $L^{-k}$ induced by $\nabla$.

It is, therefore, natural to define the operation of covariant
differentiation of distributional sections, $\forall \sigma \in\Gamma^{-\infty}(L^k\vert_U),
 \forall \xi\in\Gamma^\infty(\P\vert_U)$,
\be
\label{av_covconst23}
 (\nabla'_{{\xi}} \sigma)(\phi) := -\sigma ({\rm div}\, {\xi} \phi + \nabla^{-1}_{{\xi}}\phi),
 \quad \forall\phi \in \Gamma^\infty_c(L^{-k}\vert_U) ,
\ee
and to define the quantum space associated with the polarization as the
intersections of the kernels of the operators $\nabla'_{\ol{\xi}}:
\Gamma^{-\infty}(L^k) \to \Gamma^{-\infty}(L^k)$ for $\xi\in\Gamma^\infty(\P)$, as before. 
For K\"ahler polarizations, regularity of the Cauchy-Riemann equations (that are the
equations of covariant constancy in this case) guarantees that this definition is
conservative, that is,  one does not find distributional solutions to them which are
not holomorphic functions. On regularity of the Cauchy-Riemann equations see, for
instance, Chapter 6 in \cite{Gu} (the argument given there for the case of
distributions on $\C$ extends to $\C^n$), or \cite{KY}.

\subsection{The half-form correction} \label{av_halfform}

To define the half-form correction for the torus $\T^{2g}$ we have been considering, let us
first recall its definition for the case of a symplectic vector space $(V,\om)$ of real
dimension $2g$ (see \cite{Wo,KW}).

As before, we can for simplicity parametrize polarizations (nonnegative Lagrangian
subspaces  $\P \subset V_\C$) by the Siegel disc $\DD_g$. Over it, one considers
the canonical line bundle $K \subset \DD_g\times \La^g V_\C^* \to \DD_g$,
where each fiber is given by the space of $g$-forms that vanish upon contraction
with the conjugate of any vector in the polarization. We will use the same letter $K$
for the pull-back of this bundle to $\H_g$. The (indefinite) pairing on
the space of $g$-forms given by
\be\label{av_wedge}
 \langle \eta,\eta' \rangle \frac{(k\om)^g}{g!} = 2^{-g} \ii^g (-1)^{\frac{g(g-1)}{2}}
 \eta \wedge \overline{\eta'},
\ee
induces a pairing between any two fibers of this fibration. Notice that this pairing
is invariant under the natural action of the real symplectic group $Sp(V,\om) \cong
Sp(2g,\R)$.

Any $\Om\in\H_g$ defines a positive polarization by specifying a complex coordinate $z_\Om = x-\Om y$
on $V$, and $\dd^g z_\Om:= \dd z_\Om^1 \wedge \dots\wedge \dd z_\Om^g$
is a $g$-form generating the line of $K$ over $\Om$. A short calculation shows that the pairing comparing the
two different fibers over $\Om$ and $\Om'$ is, then,
\be \label{av_detomega}
 \langle \dd^g z_\Om, \dd^g z_{\Om'} \rangle = \det \frac{1}{2k\ii}(\Om-\ol{\Om'}) ,
\ee
and in particular we see that the pairing is positive definite over $\H_g$.

The half-form (or metaplectic) correction consists in the choice of a square root of $K$. 
Since $K$ is trivial, this looks like a trivial operation, but the behaviour of the 
inner product on the square root is an essential analytic ingredient for the BKS pairing.
We take advantage of the fact that there is a natural way of defining a specific branch
of the square root of a determinant as in (\ref{av_detomega}), given by a Gaussian integral
\[
 \left(\det \frac{1}{2k\ii}(\Om-\ol{\Om'})\right)^{-\frac{1}{2}} :=
\int_{\R^g} \e^{- \pi \frac{\t \xi (\Om-\ol{\Om'})\xi}{2k\ii}} \dd^g \xi
 = \langle \e^{\pii \frac{\t \xi\Om\xi}{2k}} ,
 \e^{\pii \frac{\t \xi\Om'\xi}{2k}} \rangle_{L^2(\R^g)} .
\]
Keeping a traditional, though possibly slightly misleading, notation $\sqrt{\dd^g z_\Om}$ for the
generator of the complex line of the half-form bundle over $\Om\in\H_g$, we define
the pairing via the embedding
\be\label{av_half-gaussian}
 \alpha \sqrt{\dd^g z_\Om} \mapsto \frac{1}{\alpha} \e^{\pii \frac{\t \xi\Om\xi}{2k}},
 \quad \alpha\in \C\setminus\{0\} ,
\ee
or, more explicitly (and with the sign determined by the Gaussian integral)
\be\label{av_ep}
 \langle \sqrt{\dd^g z_\Om}, \sqrt{\dd^g z_{\Om'}} \rangle =
 \left(\det \frac{1}{2k\ii}(\Om-\ol{\Om'})\right)^{\frac{1}{2}} .
\ee
We will see below that this point of view is also very convenient for the description of
the action of the metaplectic group on half-forms.

This definition works only over the positive Lagrangian Grassmannian;
we need to rescale the trivializing section to describe the half-form bundle
over the whole of the nonnegative Lagrangian Grassmannian. The canonical line
determined by the polarization $\P_\tau$ is generated by the $g$-form
\[
 \dd^g(x,y)_\tau = \bigwedge^g [ (1+\tau)\dd x - \ii (1-\tau)\dd y ] =
 \det (1+\tau) \dd^g z_{\Om(\tau)} ,
\]
where the last equality holds whenever $\dd^g z_{\Om(\tau)}$
% that equals $\det (1+\tau) \dd^g z_{\Om(\tau)}$ where the latter
is defined. The half-form bundle is then described by any of these two trivializations,
\[
 \C \sqrt{\dd^g z_{\Om(.)}} \cong \C \sqrt{\dd^g(x,y)_.} \to \DD_g ,
\]
where the isomorphism is given by
\be\label{av_branch}
 \sqrt{\dd^g z_{\Om(\tau)}} =  (\det (1+\tau))^{-\frac12} \sqrt{\dd^g(x,y)_\tau},
\ee
and the branch of the square root is fixed by demanding it to be $1$ for $\tau =0$.

Let us now address the half-form correction for $\T^{2g}$. Recall from Section \ref{av_tau}
that the bundle of nonnegative Lagrangian subspaces $L\T^{2g}$ of the complexified
tangent bundle is canonically trivialized,  $L\T^{2g} \cong \T^{2g} \times L\C^{2g}$.
The canonical bundle $K\to L\T^{2g}$, generated over any point $(x,y,\tau) \in
\T^{2g}\times\DD_g$ by $\dd^g(x,y)_\tau$ is again topologically trivial. It is clear
that $K$ restricted to $\T^{2g}\times\{\Om\}\cong X_\Om$ is the usual canonical bundle
$K_\Om\to X_\Om$ of holomorphic $g$-forms. As before (\ref{av_wedge}), the canonical bundle
$K$ comes with a natural Hermitian structure $h_K$ determined by the Liouville form.

As Hermitian bundle, the half-form bundle $\de \to L\T^{2g}$ is just the pull-back
\[
 \cdiag{
 \de \ar[r] \ar[d] & \C \sqrt{\dd^g(x,y)_.} \ar[d] \\
 L\T^{2g} \cong \T^{2g}\times \DD_g \ar[r] & \DD_g
 } .
\]

A connection on $\de$ is determined as follows. $K$ is equipped with a natural
partial connection $\nabla^{\mathrm{part}}$ given
pointwise at $((x,y),\tau)$ by the Lie derivative of complexified $g$-forms on
$\T^{2g}$ at $(x,y)$ along the directions of $\overline{\P_\tau}$. In fact, as $K$ is trivial, $\nabla^{\mathrm{part}}$ extends to the relative
trivial connection $\nabla^{\mathrm{triv}}$ of the family $L\T^{2g} \to \DD_g$,
corresponding to covariant derivatives along the directions of
the fibers $X_\Om$. In particular, the section $(x,y)\mapsto \dd^g(x,y)_\tau$ is
parallel relative to this connection.

The condition imposed on the connection $\widetilde{\nabla}$ on the half-form bundle $\delta$ 
is that it satisfy Leibniz's rule
\[
 \nabla^{\mathrm{triv}} (\mu\tns\mu') = (\widetilde{\nabla} \mu) \tns \mu' + \mu \tns
 (\widetilde{\nabla}\mu').
\]
Thus, the choice of connection $\widetilde{\nabla}$ is determined by the holonomies
along $2g$ generators of $H_1(\T^{2g},\R)$, subject to the condition that they square to 1.
The Hermitian structure on $\de$ is obtained by taking a consistent square root 
of the Hermitian structure on $K$.

Note that $\sqrt{\dd^g z_\Om}$ is a parallel section (along the base $\T^{2g}$, as above)
of $\de$ only if $\widetilde{\nabla}$ is the trivial connection. 
Choosing the pair $(\de,\widetilde{\nabla})$ and
an invariant holomorphic polarization $\P_\Om$ is equivalent to fixing a specific point
$\chi$ of order at most two in $\Pic^0(X_\Om)$, that is, a half characteristic. If $\chi$ is
not zero (and does not represent the trivial holomorphic bundle), there are no global
holomorphic sections; in particular, $\sqrt{\dd^g z_\Om}$, being a smooth trivialization,
cannot be
parallel along the antiholomorphic directions. Solving equations of covariant constancy on
$L^k\tns\P_\Om^*\de$ gives, then, the space of holomorphic sections $H^0(L^k\tns\chi)$, i.e.
leads to theta functions with half-characteristic.

In the remainder of the paper, for simplicity, we will use the correction
$(\de,\widetilde{\nabla})$ corresponding to $\chi=1\in \Pic^0(X_\Om)$. All the results extend
to other choices for $\chi$.

\subsection{The Weil representation of the metaplectic group}
\label{av_weilrep}

In order to study solutions of the equations of covariant constancy for polarizations on the
boundary of $\DD_g$, and in view of the Weil-Brezin expansion, it is convenient to use a
natural action of the metaplectic group on $L^2(\R^g)$. Recall from \cite{We} (see also
\cite{G,F,FN}) that the metaplectic group $Mp(2g,\R)$, the connected
two-fold cover of $Sp(2g,\R)$, can be constructed as a group of unitary operators on the
Hilbert space $L^2(\R^g)$, as follows.

Considering real $g\times g$ matrices $P,L,Q$ with $P$ and $Q$ symmetric and $L$ invertible,
and an integer $m \ \mod 4$ indexing a choice of square root $\ii^m$ of $\sign \det L$,
consider the integral operator $S(P,L,Q)_m:\S(\R^g) \to \S(\R^g)$ given by 
\be
\label{av_maurice}
 S(P,L,Q)_m f (u) :=  \ii^{-\frac{g}{2}+m} k^{\frac{g}{2}} \sqrt{|\det L|}
  \int \e^{k\pii (\t uPu -2 \t u\t Lv + \t vQv)}f(v) \dd^g v ,
\ee
where we use the notation $\ii^{\phi} := \e^{\frac{\pii}{2} \phi}$ for any real number $\phi$.
These operators are continuous on $\S(\R^g)$, therefore also on $\S'(\R^g)$, and are 
unitary isomorphisms when restricted to $L^2(\R^g)$.

The metaplectic group is the group of unitary operators generated by all operators of
this form. The covering maps each generator $S(P,L,Q)_m$ to the symplectic matrix
specified by $(P,L,Q)$. Note that this projection differs from the one in \cite{G} by
an automorphism of $Sp(2g,\R)$. Any element in $Mp(2g,\R)$ can be represented by the
product of two operators of this form. We then have,

\begin{lem}\label{av_gaussiana} Let $\Om\in\H_g$; then
\[
 (S(P,L,Q)_m \e^{k\pii \t v\Om v})(u) = \ii^m \left(\frac{|\det L|}{\det(\Om+Q)}\right)^{\frac12}
\e^{k\pii \t u\Om'u} ,
\]
where $\Om'\in\H_g$ is given by
\[
 \Om' = P-\t L(\Om+Q)^{-1} L = \left[\begin{array}{cc} PL^{-1} & P L^{-1} Q- \t L \\
 L^{-1} & L^{-1} Q \end{array}\right] (\Om).
\]
\end{lem}

\begin{proof} Applying the definitions,
\begin{eqnarray*}
 S(P,L,Q)_m \e^{k\pii \t v\Om v} (u) & = &  \ii^{-\frac{g}{2}+m}k^{\frac{g}{2}}\sqrt{|\det L|}
  \int \e^{k\pii (\t uPu -2 \t u\t Lv + \t v(\Om+Q)v)} \dd^g v \ = \\
 & = & \ii^{-\frac{g}{2}+m}k^{\frac{g}{2}} \sqrt{|\det L|}
  \sqrt{\frac{\pi^g}{\det-k\pii(\Om+Q)}} \\
 & & \quad  \e^{k\pii \t uPu} \e^{- k\pii \t u \t L (\Om+Q)^{-1} L u},
\end{eqnarray*}
from which the assertion follows.
\end{proof}

\begin{rem}
Note that the sign of the square root in the statement $\sqrt{\det(\Om + Q)}$ is
determined by evaluation of the Gaussian integral in the proof. We will not need to
specify it explicitely.
\end{rem}

The action of $Sp(2g,\R)$ on the Lagrangian Grassmannian, given by $\Omega\mapsto M(\Omega)$, lifts to 
$K$, where one obtains the ordinary pull-back of $g$-forms
\be \label{av_pbgf}
(M^*)^{-1} \dd^g z_\Om = \det(C\Om + D) \dd^g z_{\Om'},
\ee
for $M\in Sp(2g,\R)$ as in (\ref{av_mom}). By the construction of the
half-form bundle $\de$ and the identification of the Hermitian metric
(\ref{av_half-gaussian}) on it, the unitary representation of the
metaplectic group that lifts (\ref{av_pbgf}) is then given by
\be\label{av_actiononhalf}
 S(P,L,Q)_m \sqrt{\dd^g z_\Om} = \ii^{-m} |\det L|^{-\frac12}
 \sqrt{\det (\Om + Q)} \sqrt{\dd^g z_{\Om'}},
\ee 
where $S(P,L,Q)_m$ is any generator of $Mp(2g,\R)$ and $\Om' = M(\Om)$,
with $M=(P,L,Q)$.

{}From this follows immediately

\begin{prop}\label{av_descent}
The product of the action of the metaplectic group $Mp(2g,\R)$ on $\Gamma^{-\infty}(L^k)$
with the natural action on $\de$ descends to an action of $Sp(2g,\R)$
on $\Gamma^{-\infty}(L^k)\otimes\sqrt{\dd^g (x,y)_\tau}\to \DD_g$ that lifts the symplectic
action on the Siegel disc.
\end{prop}

\subsection{The extended quantum Hilbert bundle}
\label{av_eqb}

As explained above, since we consider real polarizations as limits of
holomorphic polarizations, and since for the real polarizations there are no smooth
solutions of the equations of covariant constancy, we use the weak
version of these equations and look for solutions in the space of distributional
sections $\Gamma^{-\infty}(L^k)$. This is particularly simple to do in terms of the
Weil-Brezin expansion in (\ref{av_wb}), and making use of the parameter $\tau\in\DD_g$
for the polarization in Section \ref{av_tau}.

Let $\widetilde{\Q}_\tau$ be the space of distributional sections 
$\si \in \Gamma^{-\infty}(L^k)$ that are solutions of the weak equations 
of covariant constancy (\ref{av_covconst23}) with respect to the polarization $\P_\tau$. 
These spaces naturally form a bundle of vector subspaces in the bundle
of distributional sections,
\[
 \widetilde{\Q} \subset \Gamma^{-\infty}(L^k) \times \DD_g \to \DD_g.
\]

\begin{lem}\label{av_ecc}
Under the identification of global sections of $\Gamma^{-\infty}(L^k)$ with
elements in $\S'(\R^g)^g$ via the Weil-Brezin transform, the equations of 
covariant constancy defining $\widetilde{Q}_\tau$ become independent for
each of the $g$ components $\S'(\R^g)$, and are identical on all of them.

Explicitly, using the first order differential operators $\Xi_\tau:
\S'(\R^g) \to \S'(\R^g)^g$ defined by
\be\label{av_xis}
 \Xi_\tau = (\tau+1)\pd{}{y} - 2k\pi (\tau-1)y , \qquad \tau\in\DD_g,
\ee
we have
\[
 \si \in \widetilde{Q}_\tau \iff \forall l=1,\dots,g: (\si)_l \in
 {\rm Ker}\,\, \Xi_\tau .
\]
\end{lem}

\begin{proof}
Since all the operators involved are continuous, it suffices to check the
identity on the dense subspace of smooth sections $s\in \Gamma^\infty(L^k)$.
Using (\ref{av_conn}) and (\ref{av_wb2}) we obtain
$$
(\nabla_\xi s)_l(y) = \t \beta \pd{}{y} (s)_l(y)+k2\pii \t \alpha y (s)_l(y),
$$ 
where $\xi = \t\alpha \pd{}{x}+\t \beta \pd{}{y}$ is any constant vector field.
{}From (\ref{av_polll}) we immediately get the lemma.
Note that we get the same equation for each Weil-Brezin coefficient separately.
\end{proof}

\begin{dfn} \label{av_Qdef} The quantum Hilbert bundle $\Q$ over the space of invariant
polarizations of $(\T^{2g},\om)$ is defined as 
\[
 \Q := \widetilde{\Q}\tns\C\sqrt{\dd^g(x,y)_.} \to \DD_g .
\]
where $\C\sqrt{\dd^g(x,y)_.} \to \DD_g$ is the half-form bundle.
To simplify the notation we will write $\Q_\tau = \widetilde{\Q}_\tau 
\sqrt{\dd^g(x,y)_\tau}$.
\end{dfn}

Recall that due to regularity of the Cauchy-Riemann equations the elements of the
quantum Hilbert space $\Q_\Om$ over a point $\Om$ in the interior $\breve{\DD}_g
\cong \H_g$ of the Lagrangian Grassmannian are given by the holomorphic sections
of $L^k$, with the holomorphic structure determined by $\Om$, tensored with the
half-form correction $\C\sqrt{\dd^g z_\Om}$,
\[
 \Q_\Om \cong H^0(X_\Om,L^k_\Om) \sqrt{\dd^g z_\Om} .
\]
Definition \ref{av_Qdef} extends the bundle of quantizations over the Siegel upper half-space
to the boundary of the Siegel disc $\DD_g$, proceeding as depicted in
% Figure \ref{av_fig1}, and in
the following diagram:
\[
 \cdiag{
 \Gamma^{\infty}(L^k)\sqrt{\dd^g z_{.}} \ar@{^{(}->}[r] \ar@{}[d]|\bigcup &
 \Gamma^{-\infty}(L^k)\sqrt{\dd^g (x,y)_{\tau(.)}} \ar@{}[d]|\bigcup \\
 \Q \ar@{^{(}->}[r] \ar[d] & \Q \ar[d] \\
 \H_g \cong \breve{\DD}_g \ar@{^{(}->}[r] & \DD_g
}
\]
The quantum bundle is naturally viewed as a sub-bundle of the trivial
(infinite-dimensional) bundle of smooth (on the left part of the diagram) or distributional
(on the right part of the diagram) sections of the half-form corrected prequantum bundle.
By abuse of notation, we will usually not distinguish the spaces of smooth
sections from the distributions they naturally define by integrating against
Liouville measure.

First, we prove that the dimension of the spaces $\Q_\tau$ (or $\widetilde{\Q}_\tau$)
is independent of $\tau\in\DD_g$. (Recovering a result also in \cite{An}.)

\begin{lem} For all $\tau\in\DD_g$ with $\det(1+\tau)\neq 0$, 
$$\dim \Q_\tau= k^g.$$  Explicitly, the corresponding subspace $\widetilde{\Q}_\tau$
of $\Gamma^{-\infty}(L^k)$ is spanned by the sections with Weil-Brezin coefficients
\be
 \label{av_eccsoltau}
 (\widetilde{\th}^l_\tau)_{l'} := \de_{l,l'} \det(1+\tau)^{-\frac 1 2} \e^{k\pii
\t y \Om(\tau) y } \in \S'(\R^g),
\ee
where the branch of the square root is the same one as in (\ref{av_branch}).
\end{lem}

\begin{proof}
Note that the Weil-Brezin coefficients in the statement of the lemma are
well-defined elements in $\S'(\R^g)$ even when the imaginary part of $\Om$
is only positive semi-definite, and that the dependence on $\tau$ is
continuous in this case.

By Lemma \ref{av_ecc} we have to identify the kernel of the operators
$\Xi_\tau$. Since $(1+\tau)$ is invertible and
\[
 \Xi_\tau = (\tau+1)\pd{}{y} - 2k\pi (\tau-1)y = (\tau+1)
 \left( \pd{}{y} - 2k\pii\Om y \right) ,
\]
the kernel of $\Xi_\tau$ on distributions is one-dimensional and is spanned
by the Gaussian in the statement of the lemma (see e.g. the proof of
Theorem 7.6.1 in \cite{Ho}).
\end{proof}

In order to treat the points where $\det(1+\tau)=0$, we now show that the
$Sp(2g,\R)$ action preserves the kernels of the operators $\Xi_\tau$. 

Consider the operators $ \Xi_\tau\otimes 1$ acting on $\S'(\R^g) \otimes
\delta$, where $\delta$ is the half-form bundle. Recall that, from
Proposition \ref{av_descent}, the action of the metaplectic group
$Mp(2g,\R)$ on $\Gamma^{-\infty}(L^k)\otimes \delta$ actually descends to
an action of $Sp(2g,\R)$. Note also that the group $Sp(2g,\R)$ acts
diagonally on the $g$ factors of $\S'(\R^g)$ in the Weil-Brezin expansion.

In fact, we now show that the action of $Sp(2g,\R)$ lifts to the Hilbert
quantum bundle $\Q$.

\begin{prop} Let, as above, $M=(P,L,Q)\in Sp(2g,\R)$ be a matrix with
$g\times g$ block entries $A,B,C,D$, and $\tau'=M(\tau)$; furthermore,
consider the matrix
\[
 X_{\tau'} := \frac 12\big( (\tau'+1)(A+\ii B) + (\tau'-1)(\ii C-D) \big)
\]
acting on a vector of $g$ elements of $\S'(\R^g)$. Then
\[
 (\Xi_{\tau'}\otimes 1) \circ M = X_{\tau'} \circ M \circ (\Xi_\tau \otimes 1) .
\]
In particular, the symplectic transformation $M$ preserves the kernels of the 
family of operators $\Xi_t$, which define the bundle $\Q$.
\end{prop}

\begin{proof} To shorten the expressions, set
\[
 e(u,v) := \e^{k\pii(\t u P u -2\t u \t L v + \t v Q v)} ;
\]
and calculate, using an arbitrary lifting $M_m\in Mp(2g,\Z)$ and Proposition
\ref{av_descent},
\begin{eqnarray*} && (\Xi_{\tau'}\otimes 1 \circ M)
 \left( f (u)\sqrt{\dd^g (x,y)_\tau} \right) = \\
 & = & \left((\tau'+1)\pd{M_m f}{u} - 2k\pi (\tau'-1)u M_m f(u)\right) 
 M_m (\sqrt{\dd^g (x,y)_\tau}) = \\
 & = & \int \left( (\tau'+1)\pd{e(u,v)}{u}-2k\pi(\tau'-1)u e(u,v) \right) f(v) \dd^g v
 \sqrt{\dd^g (x,y)_{\tau'}}  = \\
 & = & \int \left( (\tau'+1)2k\pii(Pu-\t Lv)-2k\pi(\tau'-1)u \right) e(u,v) f(v) \dd^g v
 \sqrt{\dd^g (x,y)_{\tau'}} .
\end{eqnarray*}
Using the fact that
\[
 u e(u,v) = L^{-1}Qv e(u,v)-\frac{1}{2k\pii}L^{-1}\pd{e(u,v)}{v}
\]
and integrating by parts, we find
\begin{eqnarray*} &&(\Xi_{\tau'}\otimes 1 \circ M)
 \left( f(u)\sqrt{\dd^g (x,y)_\tau} \right) =\\
 & = & 2k\pi\int \left[ \ii(\tau'+1)
 \left(P(L^{-1}Qv e(u,v)-\frac{1}{2k\pii}L^{-1}\pd{e(u,v)}{v})-\t Lv e(u,v)\right) - \right. \\
 & & \left. -(\tau'-1)(L^{-1}Qv e(u,v)-\frac{1}{2k\pii}L^{-1}\pd{e(u,v)}{v}) \right]
 f(v) \dd^g v \sqrt{\dd^g (x,y)_{\tau'}} \ = \\
 & = & 2k\pi\int \left[ \ii(\tau'+1)
 \left(P(L^{-1}Qv f(v)+\frac{1}{2k\pii}L^{-1}\pd{f}{v})-\t Lv f(v)\right) - \right. \\
 & & \left. -(\tau'-1)(L^{-1}Qv f(v)+\frac{1}{2k\pii}L^{-1}\pd{f}{v}) \right]
  e(u,v) \dd^g v \sqrt{\dd^g (x,y)_{\tau'}} \ = \\
 & = & \int \left[ \left( (\tau'+1)A+\ii(\tau'-1)C \right) \pd{f}{v}
 +2k\pi \left( \ii(\tau'+1)B-(\tau'-1)D \right)v f(v) \right] \\
 & &  e(u,v) \dd^g v \sqrt{\dd^g (x,y)_{\tau'}} .
\end{eqnarray*}
Therefore, to complete the proof it suffices to show that
\begin{eqnarray*}
 (\tau'+1)A+\ii(\tau'-1)C & = & X_{\tau'}(1+\tau) \\
 \ii(\tau'+1)B-(\tau'-1)D & = & X_{\tau'}(1-\tau) ,
\end{eqnarray*}
but these two equations are equivalent to the following
\begin{eqnarray*}
 (\tau'+1)(A+\ii B)+(\tau'-1)(\ii C-D) & = & 2 X_{\tau'} \\
 (\tau'+1)(A-\ii B)+(\tau'-1)(\ii C+D) & = & 2 X_{\tau'}\tau ,
\end{eqnarray*}
the first of which is the definition of $X_{\tau'}$, and the 
second one is precisely (\ref{av_changetau}).
\end{proof}

Now, we can use the metaplectic group action to show that the rank of
the bundle is constant over all of the Siegel disc $\DD_g$. Consider the
following section of $\Q\to\H_g\cong\breve{\DD}_g$,
\[
 \si^l_\Om := \th^l_\Om \sqrt{\dd^g z_\Om} =
 \widetilde{\th}^l_{\tau(\Om)}\sqrt{\dd^g (x,y)_{\tau(\Om)}} \in
 H^0(X_\Om,L^k_\Om)\sqrt{\dd^g z_\Om},
\]
where accordingly
\[
 \widetilde{\th}^l_\tau = \det (1+\tau)^{-\frac 12} \th^l_{\Om(\tau)},
\]
with $\th^l_\Om$ defined by (\ref{av_thetas}) for $\Om\in\H_g$ and where
the branch of the square root is the natural one, as in (\ref{av_branch}).

Putting things together, we obtain

\begin{thm}\label{av_teorema}
Each section $\widetilde{\th}^l_\tau$ extends continuously to a map $\DD_g\to
\Gamma^{-\infty}(L^k)$. In particular, $\tau \mapsto \{ \si^l_\tau \}_{l\in \Z^g/k\Z^g}$
is a set of global sections of $\Q$, which is therefore
trivialized and of rank $k^g$.

Furthermore, the elements of the trivializing frame $\{ \si^l_\tau \}_{l\in \Z^g/k\Z^g}$ for $\Q$ 
are invariant under the $Sp(2g,\R)$ action.
\end{thm}

\begin{proof} From Proposition 5.4.7 in \cite{Wo}, for any $\tau\in\partial\DD_g$,
there exists a symplectic transformation $M\in Sp(2g,\R)$  such that $M(\tau)$
does not have eigenvalue $-1$. Since $M^{-1}$ acts continuously on $\Q$ by the previous proposition, 
the assertion follows.

The invariance of the sections $\si^l_\tau$ under the $Sp(2g,\R)$ action is an immediate consequence of Lemma 
\ref{av_gaussiana} and of (\ref{av_actiononhalf}).
\end{proof}

\begin{rem}
For the case of reducible polarizations, the dimension $k^g$ coincides with the one
obtained by considering \'{S}niatycki's quantization procedure \cite{Sn,Ma} where one
considers smooth sections of the restriction of $L^k$ to Bohr-Sommerfeld fibers. In
fact, solutions of the weak equations of covariant constancy must be supported along 
fibers of trivial holonomy, and polarized sections are linear combinations of Dirac
delta distributions with phase variation along the Bohr-Sommerfeld fibers.
\end{rem}

\begin{ex}
\label{av_-1} 
{}For the case of the horizontal polarization, given by $\tau=-1$, the basis
$\{\sigma^l_{-1}\}_{l\in (\Z/k\Z)^g}$ consists of
distributional sections, each supported on a single compact Bohr-Sommerfeld fiber of the vertical
polarization: this follows immediately from the equations of covariant constancy written in terms of the Weil-Brezin coefficients.
Since $\tau = -1$, the operator $\Xi_{-1}$ in (\ref{av_xis}) is just multiplication by $4k\pi y$, and
its kernel is generated by the Dirac delta distribution supported at $y=0$.
Therefore, from (\ref{av_wb2}), the corresponding sections $\sigma^l_{-1}$ (each of which has a single non-zero
Weil-Brezin coefficient) will be supported on the points
with the components of $y$ being congruent to $\frac{l}{k}$. These points define the $k^g$ Bohr-Sommerfeld
fibers of the vertical polarization.
Since each basis element is supported on a single Bohr-Sommerfeld fiber, $\{\sigma^l_{-1}\}_{l\in (\Z/k\Z)^g}$  is a so-called Bohr-Sommerfeld basis
\cite{Ty} for $\Q_{-1}$.
\end{ex}

In order to describe Bohr-Sommerfeld basis for other reducible polarizations, it is convenient to study a natural geometric action
of the integer symplectic group $Sp(2g,\Z)$. For $M\in Sp(2g,\Z)$ and $\tau$ a reducible polarization, this action 
transforms Bohr-Sommerfeld leaves of the polarization $\tau$ into Bohr-Sommerfeld leaves of the 
reducible polarization $M(\tau)$. We will describe and study some properties of this action, and of its 
lift to the quantum bundle, in a forthcoming paper \cite{BMN}. In fact, that study amounts to a study of the 
algebro-geometric theta transformation formula from a symplectic point of view.

In the following Section, we study the Hilbert space structure on the fibers
of the bundle $\Q\to\DD_g$, and how they are related for different fibers.

\section{The BKS pairing}
\label{avbks}

\subsection{The BKS pairing on the extended quantum bundle}
\label{av_subsbks}

By the very construction of the quantum bundle $\Q$, over the interior of
the Siegel disc the BKS pairing is already implemented as the product of the
pairing of square integrable sections with the pairing of half-forms.
Explicitly, one has

\begin{thm} 
\label{av_obama}
For $\Om,\Om'\in\H_g$, the BKS pairing is given by
\be
 \langle \vartheta^l_\Om\tns\sqrt{\dd^g z_\Om},
 \vartheta^{l'}_{\Om'}\tns\sqrt{\dd^g z_{\Om'}} \rangle_{BKS} =
  2^{-\frac{g}{2}} k^{-g}\de_{l,l'} .
\ee
\end{thm}

\begin{proof} We have to calculate the pairing $\langle\th^l_\Om,\th^{l'}_{\Om'}\rangle$
in the prequantum Hilbert space; from the unitary isomorphism (\ref{av_unit}) between
the space of square integrable sections of $L^k$ and $(L^2 (\R^g))^k$ that restricts to
the Weil-Brezin expansion on smooth sections, and the Gaussians in (\ref{av_eccsoltau})
it is clear that this gives
\be
\label{av_nice}
 \langle\th^l_\Om,\th^{l'}_{\Om'}\rangle
= \de_{l,l'} (k\ii)^{-\frac{g}{2}} \Big( \det (\ol{\Om'}-\Om)
 \Big)^{-\frac{1}{2}} , 
\ee
which, combined with (\ref{av_ep}) proves this result.
\end{proof}

Note that from Theorem \ref{av_teorema}, the corresponding BKS pairing map 
$$B_{\Om,\Om'}:\Q_\Om\to\Q_{\Om'},$$ defined by 
\be
 \langle B_{\Om,\Om'}\sigma,\sigma'\rangle_{\Q_{\Om'}} = \langle\sigma,\sigma'\rangle_{BKS} \qquad \forall 
\sigma\in\Q_\Om, \sigma'\in\Q_{\Om'}, \nonumber
\ee
corresponds to the action on Weil-Brezin coefficients of the element of $Sp(2g,\R)$ relating $\Om$ and $\Om'$. 
It follows that the family of indexed basis
 $\{\sigma^l_\Om\}_{l\in \Z^g/k\Z^g}$ is parallel with respect to the family of pairing maps.

Moreover, from Theorem \ref{av_teorema}, the BKS pairing can be extended continuously to the
boundary of $\DD_g$, so that, $\forall \tau,\tau'\in \DD_g$, we can define unitary BKS
pairing maps 
\begin{eqnarray}
\nonumber
B_{\tau, \tau'}&:& \Q_\tau\to\Q_{\tau'}\\
B_{\tau, \tau'}(\sigma^l_\tau)&=&\sigma^l_{\tau'}, \forall l\in \Z^g/k\Z^g.
\end{eqnarray}

\begin{rem}
\label{b}
The pairing between elements of $\Q_\tau$ and $\Q_{\tau'}$, for
$\tau\in \DD_g$ and $\tau'\in \breve{\DD}_g$, as in Theorem \ref{av_obama}
is still given by evaluating the distributional sections in $\Q_\tau$ on
the conjugate of the smooth sections in $\Q_{\tau'}$ and multiplying with
the factor that arises from half-form correction. In this situation, for $\sigma\in \Q_\tau$ 
and $\sigma'\in\Q_{\tau'}$,
$$
\langle\sigma,\sigma'\rangle_{BKS} = \sigma (\ol{\sigma'}) = \ol{\sigma'} \sigma (1),
$$
where in the last term the distribution $\ol{\sigma'}\sigma$ is evaluated at the constant function $1$.
\end{rem}

Therefore, we have

\begin{cor} \label{av_p211} The family of indexed bases
\be
 \tau\mapsto\{\sigma^l_\tau\}_{l\in \Z^g/k\Z^g} \subseteq \Q_\tau, \tau\in \DD_g \nonumber
\ee
is parallel with respect to the transitive family of unitary pairing maps
$$B_{\tau, \tau'}: \Q_\tau \longrightarrow \Q_{\tau'},$$ 
$\tau,\tau'\in \DD_g$, defined by
\be
 \langle B_{\tau,\tau'}\sigma,\sigma'\rangle_{\Q_{\tau'}} = \langle \sigma,\sigma'\rangle_{BKS}
  \qquad \forall \sigma\in\Q_\tau, \sigma'\in\Q_{\tau'}. \nonumber
\ee
\end{cor}

These results show that symplectic tori also provide examples of symplectic manifolds where
quantizations in different polarizations can be related by transitive unitary BKS pairing
maps, if the half-form correction is included.

\begin{rem} From a different but related perspective, the unitary maps 
$B_{\Omega,\Omega'}$ for $\Omega,\Omega'\in \H_g$ can also be realized as coherent state 
transformations associated to different K\"ahler structures on 
the complex Lie group $(\C^*)^g$. (See \cite{FMN1} and also \cite{Hal, FMMN}.)
\end{rem}

\begin{rem} These results extend straightforwardly to non-principally polarized abelian varieties. An
interesting application is then to the moduli space of rank $n$ semistable (degree zero)
vector bundles on an elliptic curve $E_\Om$, for $\Om\in \H_1$. This moduli space is
isomorphic to ${\mathbb P}^{n-1}$ and can be realized as a quotient of the (non-principally
polarized) abelian variety $M= E_\Om \otimes \check \Lambda_R$ by the Weyl group of
$SL_n(\C)$, ${\mathbb P}^{n-1}\cong M/W$, where $\check \Lambda_R$ is the corresponding
co-root lattice. Non-abelian theta functions in genus one are then realized as $W$
anti-invariant theta functions on $M$. The unitary equivalence between spaces of
non-abelian theta functions, in genus one, associated with different complex structures
studied in \cite{AdPW,FMN2}, can therefore be formulated equivalently in terms of BKS pairing maps.
\end{rem}

\subsection{Further properties of the BKS pairing}
\label{further}

Recall that the subgroup of translations of $X_\tau$ preserving the holomorphic structure 
on $L^k$ has a central extension, given by the ``finite'' Heisenberg group $H_k$ which 
is given by
\[
H_k =  \{(\lambda,(a,b)) | \lambda\in U(1), (a,b)\in (\Z/k\Z)^{2g}\},
\]
with the group law
\[
 (\lambda,(a,b))(\lambda',(a',b'))=(\lambda\lambda' \e^{\frac{\pii}{k} (ab'-ba')}, a+a', b+b').
\]
$H_k$ acts naturally on $H^0(L^k)\cong \Q_\tau$ by
$$
(\lambda,(a,b))\sigma(x,y)=\lambda \e^{\pii \omega ((x,y),(a,b))} \sigma(x-\frac{a}{k}, y-\frac{b}{k}).
$$

This well-known natural algebro-geometric irreducible unitary representation of $H_k$ on $\Q_\tau$, 
which is unique up to isomorphism,
for $\tau \in \breve\DD_g$ (see for example Section 6.4  of \cite{BL}, or \cite{Po}) is given in the 
parallel basis by
\ba
\nonumber
(\lambda,(a,0))\sigma^l_\tau &=& \lambda \e^{- 2\pii la /k} \sigma^l_\tau\\
\nonumber
(\lambda,(0,b))\sigma^l_\tau &=& \lambda \sigma^{l+b}_\tau.
\ea

It is then natural to consider this representation for all $\tau\in \DD_g$, including 
degenerate points on $\partial \DD_g$. 

\begin{rem}
\label{av_int}From Corollary \ref{av_p211} it follows immediately  that
the unitary BKS pairing maps $B_{\tau\tau'}:\Q_\tau\to \Q_{\tau'}$ intertwine the canonical 
representations of $H_k.$
\end{rem}

We will now show that the BKS pairing between transverse real polarizations is, as expected, 
given by an intersection pairing.
Let $\tau\in \partial \DD_g$ define a real reducible polarization $\P_\tau$ and let 
$BS\subset \T^{2g}$ be the union of its Bohr-Sommerfeld fibers.  
 Recall from Remark \ref{localtoric} that if $\sigma\in \Q_\tau$ then 
the support of $\sigma$ is contained in $BS$.

\begin{prop}
\label{intersection}
Let $\tau,\tau'\in \partial \DD_g$ be such that 
$\P_\tau,\P_{\tau'}$ are real, reducible and transverse. If $\sigma\in\Q_\tau,\sigma'\in\Q_{\tau'}$ 
then the pairing $\langle \sigma,\sigma' \rangle_{BKS}$ is obtained by evaluating a distribution supported 
in $BS\cap BS'$ on the constant function $1$. 
\end{prop}

\begin{proof}
{}From Remark \ref{localtoric} and \cite{BFMN}, one has that $\sigma$ is a linear combination of 
smooth phases multiplying codimension $g$ Dirac delta distributions supported in $BS$, and similarly 
for $\sigma'$. Since, $\P_\tau$ and $\P_{\tau}$ are transverse, Theorem 8.2.4 and Example 8.2.11 
of \cite{Ho} guarantee that the product $\sigma \ol{\sigma'}$ gives a well defined distribution on $\T^{2g}$, 
supported  in $BS\cap BS'$. From the continuity of the product (see Theorem 8.2.4 in \cite{Ho}) and 
Remark \ref{b}, it is clear that $\sigma \ol{\sigma'}(1)=\langle \sigma,\sigma' \rangle_{BKS}$.  
\end{proof}

\begin{rem}
In \cite{Ma}, Manoliu has considered the quantization of the standard 
symplectic torus in real reducible polarizations at even level $k$. She introduces geometrically 
an intersection pairing between quantization spaces for two reducible (transverse) polarizations  
via the intersection points of the respective Bohr-Sommerfeld leaves. After including 
the half-form correction these pairings are unitary.
\end{rem}

\begin{rem}
Another particularly interesting point on $\partial \DD_g$ is the one corresponding to the 
vertical polarization, $\tau = 1$. 
In this case, the period matrix $\Omega =0$ and the theta 
functions become linear combinations of Dirac delta distributions
supported on Bohr-Sommerfeld fibers. Remarkably, the linear combination 
coefficients are precisely given by the modular 
transformation $S$-matrix for characters of the affine Lie algebra 
$u(1)_k$ \cite{FMN1}. Thus, in terms of Bohr-Sommerfeld basis, 
the BKS pairing map $B_{-1,1}$ is exactly represented by this modular transformation matrix. 
\end{rem}

\section{Tropical theta divisors}
\label{tropic}

Tropical geometry is known to appear in the description of degeneration data of 
boundary points in the compactification of the moduli space of polarized abelian varities 
\cite{AN}. In \cite{MZ}, the authors consider tropical curves, their Jacobians as well 
as tropical theta functions and tropical theta divisors.
Tropical geometry is also known to arise when one degenerates K\"ahler metrics. In 
this section, we will comment briefly on the appearence of tropical geometry as 
one takes the complex structure of $\T^{2g}$ to the boundary of $\DD_g$.

Recall that tropical geometry is the algebraic geometry of curves over the tropical semi-field 
$\R\cup \{-\infty\}$, where the operations are defined by $x\oplus y = \max \{x,y\}$ and 
$x\odot y = x+y$, for $x,y\in \R\cup \{-\infty\}$. The evaluation of a 
polynomial in $n$ variables over this ring defines a piecewise linear map from 
$(\R\cup \{-\infty\})^n\to \R\cup \{-\infty\}$. The associated tropical affine variety 
is then given by the set of points where this function is not smooth, which is a piecewise
linear set. 

Let us now describe several examples of metric degeneration of K\"ahler structure at a 
real reducible polarization. Namely, for computational simplicity, we will 
choose the point $\tau=-1$. We will see that the outcome depends considerably on 
the path in $\DD_g$ through which $\tau$ approaches the point $\tau = -1\in\partial\DD_g$. 

From the expression (\ref{av_JOm}) for the complex structure defined by $\Om\in\H_g$,
it is clear that the corresponding K\"ahler metric $\ga_\Om = \om(.,J_\Om.)$ is
given by the matrix
\be
\label{metric}
 \ga_\Om = \mtrx{ \Om_2^{-1} & -\Om_2^{-1}\Om_1 \\
 -\Om_1\Om_2^{-1} & \Om_1\Om_2^{-1}\Om_1+\Om_2} .
\ee

\begin{ex}
\label{ex1}
For the standard hyperbolic metric in ${\H_g}$ invariant under $Sp(2g,\R)$, 
the geodesics rays going to the point at 
infinity corresponding to $\tau =-1$ are given by
$$
\Omega (s) = B \t A + \ii A \e^{2s\Lambda}\t A, \,\,s>0,
$$
where $\Lambda$ is a positive diagonal matrix, $A\in GL_n(\R)$ and $B\t A$ is symmetric. 
Then, from (\ref{metric}),
$$
 \ga_{\Om(s)} = \mtrx{ \t A^{-1}\e^{-2s\Lambda}A^{-1} & -\t A^{-1}\e^{-2s\Lambda}\t B \\
 -B\e^{-2s\Lambda}A^{-1}  & B\e^{-2s\Lambda}\t B + A\e^{2s\Lambda}\t A} .
$$
We see that, after rescaling the metric appropriately, 
as $s\to\infty$ at least $g$ dimensions collapse. In fact, the number of 
surviving dimensions equals the multiplicity of the largest eigenvalue of $\Lambda$.

Let us therefore consider the case when $\Lambda = \lambda 1$.
Then the rescaled metrics 
$ \e^{-2s\lambda}\ga_{\Om(s)}$ converge in the Gromov-Hausdorff sense,
\[
 \left( \T^{2g}, \e^{-2s\lambda} \ga_{\Om(s)} \right) \to
 (\T^g, A \t A ) .
\]

One may ask how theta divisors behave under this metric degeneration.
For the sake of simplicity, consider for instance the theta divisor $V(\th_\Om)$ for
level $k=1$. Let $u=\e^{2s\lambda}$. The absolute values of the terms of the
series defining $\th_{\Om(s)}$ are
\[
 a_m(y) = \e^{-\pi u \t(y-m)A\t A(y-m)},
 \quad m \in \Z^g.
\]
Call $m\in \Z^g$ a lattice neighbor of $y$ if
\[
 \forall m'\in\Z^g\setminus\{m\}: \quad
 \| y-m \|_{A\t A} \leq \| y-m' \|_{A\t A}.
\]
As $A\t A$ is positive definite, it is easy to see that for
any point $y_0$ which does not have at least two distinct lattice neighbors,
the theta function cannot equal zero for $s$ large enough and
any value of $x$, i.e.
\[
 y_0 \notin \mu(V(\th_{\Om(s)})) ,
\]
where $\mu$ stands for the group-valued moment map
\[
 \T^{2g}\ni(x,y)\mapsto \mu(x,y)=y\in\T^g.
\]
The theta divisor then approaches the tropical theta divisor of \cite{AN,MZ}
(see Section 5.2 of \cite{MZ}), defined by the non-smoothness locus of the functions
$$
 \max_{m\in\Z^{g}} \{ -\t(y-m)A\t A(y-m) \}
 \quad {\rm or} \quad
\max_{m\in\Z^{g}} \{ -\t mA \t Am +2\t yA \t Am \}.
$$
Note that this set depends only on the limit metric and on the location of the 
lattice points which define the Bohr-Sommerfeld fiber. This behaviour of theta
divisors as one approaches the boundary of $\DD_g$ is consistent with the fact 
that the theta function approaches a distributional section supported on the 
Bohr-Sommerfeld fiber, so that its zeroes are away from this fiber. 
\end{ex}

In general, the limiting behaviour of the theta divisor will not be related to 
the limit metric in such a simple manner as above. 

\begin{ex}
Let $s>0$. Considering a half-line of 
complex structures $\Om+s\dot{\Om}$ with $\Om,\dot{\Om}
\in\H_g$ and $s\to\infty$, the rescaled metrics
$\frac{1}{s} \ga_{\Om+s\dot{\Om}}$ converge in the Gromov-Hausdorff sense,
\[
 \left( \T^{2g}, \frac{1}{s} \ga_{\Om+s\dot{\Om}} \right) \to
 (\T^g, \dot{\Om}_1\dot{\Om}_2^{-1}\dot{\Om}_1+\dot{\Om}_2 ) .
\]
Note that the point on the boundary of $\DD_g$ we are approaching (or
the real polarization) is still the same, $\tau=-1$.

A computation analogous to the one in Example \ref{ex1} shows that the limit of the 
theta divisors is the same as above for the metric $\dot{\Om}_2$. This coincides with 
the limit metric only in the case $\dot{\Om}_1=0$, which is when 
the half-line is a (reparametrized) geodesic.
\end{ex}

\begin{ex}
For higher level $k>1$, a particular tropical theta divisor is given in an  
analogous way by the non-smoothness locus of
$$
y \mapsto
\max_{m\in\Z^{g}} \{ -\t m\dot{\Om}_2m +2k\t y\dot{\Om}_2m \}.
$$
This piecewise linear object, which is equidistant from Bohr-Sommerfeld fibers, 
is obtained by degeneration of the divisor of the theta 
function $\sum_{l\in \Z^g/k\Z^g} \sigma_\tau^l$, which is invariant by the 
subgroup $\Z^g/k\Z^g \subset H_k$ generated by the elements of the form $(1,(0,b))$.
(See Section \ref{further}.) Similar objects, the geometric quantization 
amoebas of \cite{BFMN}, appear in the geometric quantization of toric manifolds, 
where also such a particular tropical divisor, equidistant from Bohr-Sommerfeld points  
in the polytope, (asymptotically) selects a particular degenerating holomorphic section.  
\end{ex}

\vspace{2ex}

\textbf{Acknowledgements:} We wish to thank C.Florentino for useful 
discussions, J.Drumond Silva for help with the proof of 
Proposition \ref{intersection}
and C.Baier for help with the literature. We also wish to thank
the referee for important suggestions. 
Partially supported by the Center for Mathematical Analysis, Geometry and 
Dynamical Systems, by Funda\c{c}\~ao para a Ci\^encia e a Tecnologia through the 
Program POCI 2010/FEDER and by the Projects POCI/MAT/58549/2004 and
PPCDT/MAT/58549/2004. TB is also supported by the fellowship
SFRH/BD/22479/2005 of Funda\c{c}\~ao para a Ci\^encia e a Tecnologia.

%%%%%%%%%%%%%%%%%%%%%%%%%%%% References %%%%%%%%%%%%%%%%%%%%%%%

\end{document}